\documentclass[12 pt,letterpaper]{article}
\usepackage{amsmath,amsthm,amssymb,amsfonts,graphicx,psfrag,fancybox,indentfirst}
\usepackage{hyperref}

\usepackage{bbm}

\theoremstyle{plain}
\newtheorem{theorem}{Theorem}[section]                                          
\newtheorem{proposition}[theorem]{Proposition}                          
\newtheorem{lemma}[theorem]{Lemma}
\newtheorem{corollary}[theorem]{Corollary}

\theoremstyle{definition}
\newtheorem{definition}[theorem]{Definition}
\theoremstyle{remark}

\begin{document}


\title{Generalized Dirichlet distributions on the ball and moments}
\large
\author{F. Barthe\footnote{Equipe de Statistique et Probabilit\'es, 
Institut de Math\'ematiques de Toulouse (IMT) CNRS UMR 5219, 
Universit\'e Paul-Sabatier, 
31062 Toulouse cedex 9, France. e-mail: franck.barthe@math.univ-toulouse.fr}, F. Gamboa\footnote{IMT e-mail: fabrice.gamboa@math.univ-toulouse.fr}, L. Lozada-Chang\footnote{
Facultad de Matem\'atica y Computaci\'on, 
Universidad de la Habana, 
San L\'azaro y L, Vedado
10400 C.Habana, Cuba. e-mail: livang@gmail.com} and A. Rouault\footnote{Universit\'e de Versailles, LMV	
B\^atiment Fermat,  	
45 avenue des Etats-Unis,  	
78035 Versailles cedex, France. e-mail: alain.rouault@math.uvsq.fr	
}}
\setcounter{tocdepth}{3}

\thispagestyle{empty}
\maketitle

\begin{abstract}
  The geometry of unit $N$-dimensional $\ell_{p}$ balls (denoted here by $\mathbb{B}_{N,p}$) has been intensively investigated in the past decades. A particular topic of interest has been the study of the asymptotics of their projections. Apart from their intrinsic interest, such questions have applications in several probabilistic and geometric contexts \cite{Barthe1}. In this paper, our aim is to revisit some known results of this flavour with a new point of view. Roughly speaking, we will endow $\mathbb{B}_{N,p}$ with some kind of Dirichlet distribution that generalizes the uniform one and will follow the method developed in \cite{Skibinsky}, \cite{Chang} in the context of the randomized moment space. The main idea is to build a suitable coordinate change involving independent random variables. Moreover, we will shed light on  connections between the randomized balls and the randomized moment space.

\end{abstract}


\normalsize\rm
\bigskip

\section{Introduction}
\label{sintro}
The starting point of our work is the study of the asymptotic behaviour of the moment spaces:
\begin{equation}
     M_N^{[0,1]} = \left\{\left(\int_0^1 t^j \mu(dt)\right)_{1\leq j\leq N} : \mu \in {\mathcal{M}_1([0,1])}\right\},
     \label{momo1}
\end{equation}
where ${\mathcal{M}_1([0,1])}$ denotes the set of all probability measures on $[0,1]$.   These compact sets are randomized with the uniform distribution. By using cleverly an old Skibinsky's result (\cite{Skibinsky}),  the authors of the seminal paper
\cite{Chang}  show two very nice results. First they proved that,  for large $N$, the sets $(M_N^{[0,1]})$ are {\it almost} concentrated in terms of finite dimensional projections on one point (the moments of the arcsine law). Secondly, they obtained a multidimensional CLT for the fluctuations of the finite dimensional projections. The asymptotic covariance matrices only involve the moments of the arsine law. These results have been extended to large deviations asymptotics in \cite{FabLoz}. The main tool for the study of randomized moment spaces is the existence of a {\it nice} coordinate change  leading to independent random variables. This coordinate change is obtained by the so-called Knothe map (see the proof of the first theorem p.40 in \cite{Knothe}) that is available for any finite dimensional bounded convex body. The new coordinates are called canonical moments in the literature. We refer to the excellent book \cite{DS} for a complete overview on canonical moments.  
Notice that recently these results have been extended to matricial moment spaces (see \cite{DS1}).\\
\ \\
In this paper we will focuss both on randomized $N$-dimensional $\ell_{p}$ balls 
(denoted by $\mathbb{B}_{N,p}$) and randomized moment spaces. First the randomized ball will be studied by using the Knothe map. Surprisingly, as in the case of moment spaces, when the ball is endowed with the uniform distribution, the Knothe map leads to canonical coordinates that are also independent random variables.
Furthermore,  we will show that this property remains true for a general family of distributions on the ball (called $p$-generalized Dirichlet distribution, see Section \ref{susuge}). Independence will be the main tool to investigate the properties of the randomized balls. Indeed, it will enable to easily show general association results for the $p$-generalized Dirichlet distribution (see Section \ref{susuoth}). Furthermore, with the help of the canonical coordinates we will study various  Poincar\'e-Borel like lemmas for these distributions. That is, convergence and fluctuations of the projections when the dimension of the space increases (see Section \ref{sec:asymptotics}).\\
\ \\
There is a revival interest on the moment problem and on orthogonal polynomials on the torus $\mathbb{T}$. It is also possible in this frame to define canonical moments. They are also sometimes called Verblunsky coefficients (see for example \cite{Simon1}). Notice that these coefficients have a lot of properties. For example they are involved in the inductive equations of orthogonal polynomials construction. In this paper, we will discuss  connections between randomized  moment spaces (for moment problems on $\mathbb{T}$), and the ball $\mathbb{B}_{N,p}$.  This connection will be obtained through the canonical coordinates (see Section \ref{ssec:mom_spc_GD}).\\
\ \\
The paper is organized as follows.
In the next section, we recall some definitions and useful properties of Dirichlet distributions. We also recall the definition and some basic properties of the canonical moments on a compact interval. Then, building the same parametrization for $\mathbb{B}_{N,p}$ we introduce and study generalized Dirichlet distributions on $\mathbb{B}_{N,p}$ (called the $p$-generalized Dirichlet distributions).  The uniform distribution appears to be a special case.
We also discuss the connections between these results and the so-called stick-breaking construction of the Dirichlet distributions. We end the section settling negative association properties for these distributions.
In Section~\ref{sec:asymptotics},  we obtain asymptotic results for $p$-generalized Dirichlet distributions. 
Let us notice that there is an extension of generalized Dirichlet distributions in the context of matrix balls (see \cite{Nene}).

We give several applications of the representation on independent variables proving several versions of the Poincar\'e-Borel lemma (see e.g. \cite{Ledoux}) working both with weak convergence and large deviations.
Finally in the last Section, we discuss some connections between randomized balls and randomized moments spaces.


\section{Probabilities for moment sets and balls}

\subsection{The Dirichlet world}
\label{sec.dirichlet.world}

Let us recall some useful properties and definitions related to Dirichlet distributions (see for example \cite{Ko00}).
A large class of laws on the unit ball may be built from the Dirichlet distributions.

\label{direlem}
We use two definitions of simplices. For $k \geq  1$, we set
\begin{eqnarray*}
   \mathcal{S}_{k+1} &=& \{(x_1, \cdots, x_{k+1}) : x_i > 0 , (i = 1, \cdots , k+1) ,  \ x_1 + \cdots + x_{k+1}  = 1 \},\\
   \mathcal{S}_k^< &=& \{(x_1, \cdots, x_k) : x_i > 0 , (i = 1, \cdots , k) ,  \ x_1 + \cdots + x_k  < 1 \}.
\end{eqnarray*}

It is clear that the mapping $(x_1, \cdots , x_{k+1}) \mapsto (x_1, \cdots, x_k)$ is bijection from $\mathcal{S}_{k+1}$ onto  $\mathcal{S}_k^<$.

For $a_{1},\cdots,a_{k+1}>0$, the Dirichlet distribution $\hbox{Dir}(a_1, \cdots, a_{k+1})$ on $\mathcal{S}_{k+1}$
has the density
\begin{equation*}
\label{defdir}
    f(x_1, \cdots, x_{k+1}) = \frac{\Gamma(a_1 + \cdots + a_{k+1})}{\Gamma(a_1) \cdots \Gamma(a_{k+1})} x_1^{a_1 - 1} \cdots x_{k+1}^{a_{k+1} -1},
\end{equation*}
with respect to the Lebesgue measure on $\mathcal{S}_{k+1}$.
It can also be viewed as a distribution Dir$_k(a_1, \cdots, a_k ; a_{k+1})$ on $\mathcal{S}_k^<$ with density
\begin{equation*}
\label{defdir2}
    f^<(x_1, \cdots, x_k) = \frac{\Gamma(a_1 + \cdots + a_{k+1})}{\Gamma(a_1) \cdots \Gamma(a_{k+1})} x_1^{a_1 - 1} \cdots x_k^{a_k -1} \left(1 - x_1 - \cdots -x_k\right)^{a_{k+1}-1}.
\end{equation*}
The particular case $a_1= \cdots =  a_{k+1}=1$ is the uniform distribution on $\mathcal{S}_k^<$.

We recall that the family of Dirichlet distributions is stable by {\it partial sum grouping}, i.e.
if $(\sigma_1 , \cdots , \sigma_m)$ is a partition of $\{1, \cdots, k+1\}$,  then
\begin{equation}
\label{grouping}
    (x_1, \cdots, x_{k+1}) \stackrel{d}{=}\  \hbox{Dir}(a_1, \cdots, a_{k+1}) \Rightarrow (X_1, \cdots, X_m) \stackrel{d}{=}\  \hbox{Dir}(A_1, \cdots, A_{m})
\end{equation}
where $X_i = \sum_{j\in \sigma_i} x_j$ and $A_i = \sum_{j\in \sigma_i} a_j$.
Moreover
\begin{equation}
\label{quot}
   \left(\frac{X_1}{X_1+\cdots +X_{m-1}} , \cdots , \frac{X_{m-1}}{X_1+\cdots +X_{m-1}}\right) \stackrel{d}{=}\
   \hbox{Dir}\left(A_1, \cdots, A_{m-1}\right)
\end{equation}
and this vector is independent of $X_1+\cdots +X_{m-1}$.

If $a_1= \cdots = a_{k+1} = a$, we denote the distribution by $\hbox{Dir}_k (a)$. If $k=2$ the distribution Dir$_1(a ; b)$ on $(0,1)$ is the Beta$(a,b)$ distribution.
Sometimes we need the affine push forward  on $(-1, +1)$ of  this distribution. We denote this last distribution by $\mathrm{Beta}_s(a,b)$. 
To end this preliminary let us recall the classical relation between Gamma and Dirichlet distributions. For $a,\lambda>0$, we say
that $Z\sim\gamma(a,\lambda)$ whenever its distribution has the following
density
$$h_{a,\lambda}(y):=\frac{y^{a-1}\lambda^a}{\Gamma(a)}\exp(-\lambda y),\;\;
(y>0).$$
We use frequently the slight abuse of notation $\gamma(a) = \gamma(a,1)$.

If $y_i$, $i =1,2, \cdots , r$ are independent and if $y_i \stackrel{d}{=}\  \gamma(b_i)$ then
\begin{equation}
\label{10}
\left(\frac{y_1}{y_1+\cdots +y_{r}} , \cdots , \frac{y_{r}}{y_1+\cdots +y_{r}}\right) \stackrel{d}{=}\  \hbox{Dir}\left(b_1 , \cdots , b_{r}\right)
\end{equation}
and this variable is independent of  $y_1+\cdots +y_{r}$.
It is a generalization of the well known fact
\begin{equation*}
   \mathrm{Beta}(a,b) \stackrel{d}{=}\  \frac{\gamma(a)}{\gamma(a) + \gamma'(b)}\,,
\end{equation*}
where $\gamma(a)$ and $\gamma'(b)$ are independent.\\
Let ${\mathbf G}_p$ be the distribution on $\mathbb{R}$ with density
$$x \mapsto \frac{1}{2\Gamma\left(1 +\frac{1}{p}\right)} \!\ e^{-|x|^p}\,.$$
It is the distribution of $\varepsilon Z^{1/p}$ where $Z$ has the  $\gamma(p^{-1})$ distribution, and $\varepsilon$ is a Rademacher variable ($\mathbb{P}(\varepsilon=\pm 1)=0.5$) independent of $Z$.
\subsection{Stick-breaking and generalized Dirichlet distributions}
\label{subeta}
The following classical model has been widely used in geometric probability,  genetics, Bayesian statistics, etc.... It leads to the fascinating area of random distributions (see Kingman \cite{Kingrdd} and Pitman \cite{Pit}). It is often known as stick-breaking. For the sake of consistency, let us explain the details of the construction.

We define two sets of variables $(Z_1, \cdots, Z_n) \in (0,1)^n$ and $(P_1, \cdots, P_n)\in \mathcal{S}_n^<$ connected by the system of equations
\begin{equation}
   \label{ptoz}
   \begin{aligned}
   Z_1 &= P_1\\
   Z_j &= P_j \left(1 - P_1 - \cdots -P_{j-1}\right)^{-1}\ \ , \ j= 2, \cdots n\,.
   \end{aligned}
\end{equation}
which is equivalent to
\begin{equation}
   \begin{aligned}
   \label{ztop}
   P_1 &= Z_1\\
   P_j &= Z_j \left[\prod_{k=1}^{j-1} (1- Z_k)\right]\ \ , \ j= 2, \cdots n\,.
   \end{aligned}
\end{equation}
We may add $Z_{n+1} = 1$ which is equivalent to $P_{n+1} = 1 - P_1 - \cdots - P_n$.

It is possible to define the infinite model: $(Z_j)_{j \geq 1}$ and $(P_j)_{j \geq 1}$ connected by (\ref{ptoz}) and (\ref{ztop}). In that case, $\sum_j P_j = 1$ is  equivalent to $\prod_j (1-Z_j) = 0$.
Actually,  it can be thought as  a sequential procedure to generate an element of $\mathcal{S}_n$ (or $\mathcal{S}_\infty$) viewed as a partition of $[0,1]$ into segments. The value $P_1 = Z_1$ gives a bisection $[0, P_1]\cup (P_1, 1]$ of $[0,1]$.  To the rightmost segment we perform a new bisection in proportion $Z_2$, so that $(P_1, 1]$ gives $(P_1, P_2]\cup (P_2, 1]$ with $P_2 = Z_2(1-P_1)$, and so on.

They are essentially two ways to provide these variables with probability distribution, starting either  from the $Z$'s or  from the $P$'s. The common feature of all popular randomizations is the independence of the $Z$ variables.
In the elementary model $Z_j$ is uniform on $[0,1]$ for every $j\leq n$. The model was extended successively to $Z_j \stackrel{d}{=}\  \mathrm{Beta}(1, \theta)$ with  $\theta > 0$ (it is the so called GEM$(\theta)$ model), and later to $Z_j \stackrel{d}{=}\  \mathrm{Beta}(1-\alpha, \theta + j\alpha)$ for $\theta > -\alpha$ and $0< \alpha < 1$ (it is the so called GEM$(\alpha, \theta)$ model).  The bibliography in \cite{fengseul} is rather extensive.
Besides, for biological and Bayesian statistical motivations, Connor and Mosimann \cite{connor1969cip}, assumed
\begin{equation*}
\label{oconnor}
Z_j \stackrel{d}{=}\  \mathrm{Beta}(a_j, b_j)\ \ , \ j=1 , \cdots , n
\end{equation*}
with $Z_{n+1} = 1$, and where $a_1, \cdots, a_n, b_1, \cdots, b_n$ are positive numbers.
They noticed (formula 14 p.199) that the density of $P= (P_1, \cdots, P_n)$ on $\mathcal{S}_n^<$ is
\begin{equation*}
   \mathcal{GD}_{\bf a, \bf b}(p_1, \cdots, p_n) := \frac{1}{\mathcal{Z}({\bf a}, {\bf b})} \ p_{n+1}^{b_n -1} \prod_{j=1}^n \left[p_j^{a_j -1} \left(1 - p_1 - \cdots -p_{j-1}\right)^{b_{j-1}- (a_j + b_j)}\right]
\end{equation*}
where $p_{n+1} = 1 - p_1 - \cdots -p_n$ and
\begin{equation*}
   \mathcal{Z}({\bf a}, {\bf b}) = \prod_j \frac{\Gamma(a_j)\Gamma(b_j)}{\Gamma(a_j+b_j)}
\end{equation*}
is the normalizing constant.

They called this distribution the Generalized Dirichlet distribution of parameters ${\bf a} = a_1, \cdots a_n$ and ${\bf b} = b_1, \cdots , b_n$.
We recover the Dirichlet distribution Dir$_n(a_1, \cdots, a_n ; b_n)$ when the parameters satisfy the relations $b_{j-1} = a_j + b_j$ i.e.
\begin{equation}
   \label{relab}
   b_j = a_{j+1} + \cdots + a_{n}+ b_n \ \ j = 1, \cdots , n-1\,.
\end{equation}

The two following properties are consequences of the construction (\ref{ztop}) (see also \cite{wong1998gdd}):

1)  $P^{(k)} = (P_1, \cdots, P_k) \stackrel{d}{=}\  \mathcal{GD}_{\bf a^{(k)}, \bf b^{(k)}}$

2)  For every $k = 1, \cdots n-1$, conditionally upon $Z_1 , \cdots Z_k$,
\begin{equation*}
 \left(\frac{P_{k+1}}{1- P_1 - \cdots -P_k} , \cdots , \frac{P_{n}}{1- P_1 - \cdots -P_k}\right) \stackrel{d}{=}\  \mathcal{GD}_{a_{k+1}, \cdots, a_n, b_{k+1}, \cdots, b_n}
\end{equation*}

All the above models, where we provide each $Z_j$ with a Beta distribution with prescribed parameters, yield $\mathcal{GD}$ distribution for the corresponding vector $P$.
Conversely, it is known that if  $P$ is uniformly distributed on the simplex $\mathcal{S}_n^<$, then $Z_j$ is Beta$(1, n-j+1)$ distributed for $j \leq n$.
The $\mathcal{GD}$ distribution has a more general covariance structure than the Dirichlet distribution.

In the Section~\ref{ssec:samplig_ball} we carry out the same construction for $\ell^p$ ball.

\subsection{Real canonical moments}

In this section, we recall some interesting objects related to  moment spaces. In \cite{Knothe}, aiming to extend Brunn-Minkowki's theorem to convex bodies, Knothe introduced a general coordinate change. Skibinsky \cite{Skibinsky} used this tool in the context of moment spaces. His goal was the study of some geometric aspects of these sets. In this context the new coordinates are called canonical moments. These quantities  play an important role in moment problem theory. Indeed, they appear in many topics such as the orthogonal polynomial recurrence relation, the Stieltjes transform (and its expansion in continued fraction), etc... Actually the canonical moments seem to be more intrinsically related to the probability measures than the algebraic moments. In Section~\ref{sec.momens.complex} we present the canonical moments for complex moment space.
Although a geometric construction is possible we define them using orthogonal polynomials on the unit complex circle following Simon in \cite{Simon1}. We refer to the excellent book of Dette and Studden \cite{DS} for a complete overview on canonical moments. In next section we will carry the same geometric construction to $\ell_p$ balls.
Recall that we denote  by ${\mathcal{M}_1([0,1])}$ the set of all probability measures on $[0,1]$  and by $M_N^{[0,1]}$ the $N$-th algebraic moment space generated by probability measures on $[0,1]$ (see \ref{momo1}).
Let $\mu \in \mathcal{M}_1([0,1])$,  we define, for $n \geq 1$,
\begin{eqnarray*}
   c_{n+1}^+\left(\mu\right) &=& \max\left\{r \in \mathbb{R} : (m_1, \cdots , m_n, r) \in M_{n+1}^{[0,1]}\right\}\\
   c_{n+1}^-\left(\mu\right) &=& \min\left\{r \in \mathbb{R} : (m_1, \cdots , m_n, r) \in M_{n+1}^{[0,1]}\right\},
\end{eqnarray*}
where $(m_1,m_2,...,m_n)$ is the vector of $n$ first moments of $\mu$. The first canonical moment is $c_1 = m_1$ and, for $n \geq 1$, the $n+1$-th canonical moment is defined as
\begin{equation*}
   \label{defcanonr}
   c_{n+1}(\mu) =  \frac{m_n - c_{n+1}^-(\mu)}{ c_{n+1}^+(\mu) - c_{n+1}^-(\mu)}
\end{equation*}
whenever $c_{n+1}^+(\mu) > c_{n+1}^-(\mu)$. The last condition is verified if, and only if, 
\[(m_1,m_2,...,m_n)\in\mathrm{int} M_{n}^{[0,1]}\,\]
 Obviously, the canonical moments depend on $\mu$ just through its moment vector. Thus, given  $(m_1,m_2,...,m_n)\in \mathrm{int} M_{n}^{[0,1]}$  the vector of $n$ first canonical moments is completely defined. 
Furthermore, the mapping $(m_1,m_2,...,m_N) \mapsto (c_1,c_2,...,c_N)$ from $\mathrm{int}M_N^{[0,1]}$ onto $(0,1)^N$ is bijective and triangular in the sense that for every $k\leq N$, $c_k$ depends only on $(m_1,m_2,...,m_k)$ and not on $m_{k+1} , \cdots , m_N$. 
%
The {\it range} sequence is given by the following relation due to Skibinsky  (see \cite{Skibinsky} or Theorem 1.4.9 in \cite{DS})
\begin{equation*}
\label{range}
    c_{i}^+(\mu)- c_{i}^-(\mu) = \prod_{j=1}^{i-1}c_j(\mu)\left(1 - c_j(\mu) \right),\ i=2,3,...,N.
\end{equation*}
From this relation, it follows that the bijective mapping $m \mapsto c$ from int $M_N^{[0,1]}$ onto $(0,1)^N$ is a diffeomorphism whose Jacobian is
\begin{equation}\label{realjacob}
      \frac{\partial (m_1 \cdots, m_N)}{\partial (c_1, \cdots, c_N)}= \prod_{j=1}^{N-1} \left(c_j (1- c_j)\right)^{N-j}\, .
\end{equation}

The Jacobian in \eqref{realjacob} leads to the following result.
\begin{theorem}[Chang, Kemperman, Studden]\label{thm:chang}
\label{unifreal}
If $M_N^{[0,1]}$ is endowed with the uniform distribution, then the random canonical moments $C_j$ $j = 1, \cdots, N$ are independent and
$$\mathbb{P} (C_j \in dx) = \frac{(2N-2j +1)!}{\left((N-j)!\right)^2} \ x^{N-j}(1-x)^{N-j} {\mathbf 1}_{(0,1)}(x)\ dx.$$
\end{theorem}
It is Theorem 1.3 in \cite{Chang} (see also p. 305 in \cite{DS}). In other words $C_j$ is $\mathrm{Beta}(N-j+1, N-j+1)$ distributed.

\subsection{Canonical coordinates on the unit real ball}
\label{canonr}
Following the last geometric definition of canonical moments in $M_N^{[0,1]}$, it is possible to build similar quantities on the $\ell_p$ ball. To begin with,
let for $N\geq 1$ and $1\leq p<\infty$ let $\mathbb{B}_{N,p}$ (resp. $\mathbb{B}^\mathbb{C}_{N,p}$)
be the the unit $p$-ball of $\mathbb{R}$,  (resp. on $\mathbb{C}$). That is,
\begin{equation*}
     \mathbb{B}_{N,p}:= \left\{x=(x_1,x_2,...,x_N)\in \mathbb{R}^N :\|x\|_p:= \left(\sum_{i=1}^{N} |x_i|^p\right)^{1/p} < 1\right\}.
\end{equation*}
Extending the previous notation, $\mathbb{B}_{N,\infty}$ will denote the unit open $l^{\infty}$ ball, i.e. $(-1,1)^N$.
For $x=(x_1,x_2,...,x_N)\in\mathbb{R}^{N}$ and $1\leq k\leq N$, set $x^{(k)}=(x_1,x_2,...,x_k)$, i.e.
the subvector of  $x$ built by its first $k$ coordinates. By convention,
we set $\|x^{(0)}\|_p:=0$. We introduce the \textit{canonical coordinates} $c=(c_{1},\ldots,c_{N})\in (-1,1)^N$ of $x\in\mathbb{B}_{N,p}$
\begin{equation}
   \begin{aligned}
   c_1 &= x_1,\label{unun}
   \\
   c_k &= \frac{x_k}{\sqrt[p]{1-\|x^{(k-1)}\|_p^p}},\quad k=2,3,\ldots,N.
\end{aligned}
\end{equation}
Knowing $x^{(k-1)}$, we see that $\left(-\left(1-\|x^{(k-1)}\|_p^p\right)^{1/p}, \left(1-\|x^{(k-1)}\|_p^p\right)^{1/p}\right)$
is the admissible range of $x_k$ in order that $x$ lies in $\mathbb{B}_{N,p}$.
Let us denote by  $\mathcal{C}_N$ the mapping from $\mathbb{B}_{N,p}$ onto $(-1,1)^N$ which associates to any $x\in \mathbb{B}_{N,p}$ the point $c=(c_1,\cdots,c_N)$ defined in \eqref{unun}.
The following key property is straightforward.
\begin{lemma}
\label{lemjacob}
 The mapping $\mathcal{C}$ is a triangular $C^1$-diffeomorphism, its inverse is given by
 \begin{equation*}
   [\mathcal{C}^{-1}(c)]_k = c_k\sqrt[p]{(1-|c_1|^p)(1-|c_2|^p)\cdots(1-|c_{k-1}|^p)},\quad
      k = 1,2,\cdots,N,
   \label{oninverse}
\end{equation*}
 the Jacobian matrix is lower triangular and its determinant is
\begin{equation*}
   \label{Jacob}
   \frac{\partial (x_1, \cdots, x_N)}{\partial (c_1, \cdots , c_N)}  =\prod_{k=1}^{N} (1-|c_k|^p)^{\frac{N-k}{p}}, \quad
         (c_1,\cdots,c_N)\in (-1, 1)^N\,.
\end{equation*}
\end{lemma}

\subsection{Extension to the unit complex ball}
\label{complexball}

We extend the previous construction to the complex framework. First define
\begin{equation}\label{defboulecomplex}
     \mathbb{B}_{N,p}^{\mathbb{C}}:=
     \left\{z=(z_1,z_2,...,z_N)\in \mathbb{C}^N :\|z\|_p=
     \left(\sum_{i=1}^{N} |z_i|^p\right)^{1/p} < 1\right\}.
\end{equation}
The canonical coordinates are now
\begin{equation*}
\begin{split}
   c_1 &= z_1,
   \\
   c_k &= \frac{z_k}{\sqrt[p]{1-\|z^{(k-1)}\|_p^p}},\quad k=2,3,\ldots,N
\end{split}
\end{equation*}
where for $z=(z_1,z_2,...,z_N)$, as in the real case, we set $z^{(k)} = (z_1,z_2,...,z_k)$, $k=1,2,...,N$.
Now the canonical coordinates belong to $\mathbb{D}$. Conversely we have
\begin{equation}
   z_k = c_k\sqrt[p]{(1-|c_1|^p)(1-|c_2|^p)\cdots(1-|c_{k-1}|^p)}\quad
   k=1,2,\cdots,N.
   \label{oninversec}
\end{equation}
We now point up the case $p=2$. Consider the bijection $T:\mathbb{C}^N \to \mathbb{R}^{2N}$,
\begin{equation} \label{isometry}
     T(x_1+iy_1,...,x_N+iy_N)  = (x_1,y_1,...,x_N,y_N).
\end{equation}
Only for $p=2$ we have $T\left(\mathbb{B}_{N,p}^{\mathbb{C}}\right) = \mathbb{B}_{2N,p}$. However, the image of the canonical coordinates of $z \in \mathbb{B}_{N,2}^{\mathbb{C}}$ by the previous bijective map are not, in general, equal to the canonical coordinates of the image of $z$. 


\subsection{Sampling on the $\ell^p$ ball}
\label{ssec:samplig_ball}
\subsubsection{Simplex and balls}
The simplest way to sample in the ball is to use the uniform distribution. So that, the stick-breaking scheme gives  a first method to sample uniformly in $\mathbb{B}_{N,p}$.  As a matter of fact, it is enough to sample independent beta random variables and perform the change $\mathcal{C}_N^{-1}$.
This procedure is inherited from the sampling Dirichlet distribution as proposed in the book of Devroye (Theorem 4.2 p. 595 of \cite{luc} chapter 11 - see also the error file).
The following lemma gives a first connection with the simplex.
\begin{lemma}
\label{xixi0}
If $X := (X_1, \cdots, X_N)$ is uniformly distributed in $\mathbb{B}_{N,p}$, then
\begin{enumerate}
\item
\begin{equation*}
(X_1, \cdots , X_N) \stackrel{d}{=}\  \left(\varepsilon_1 \rho_1^{1/p}, \cdots , \varepsilon_N \rho_N^{1/p}\right)
\end{equation*}
where
$(\rho_1, \cdots, \rho_N)$ follows the  $\hbox{Dir}_N(p^{-1}, \cdots, p^{-1} ; 1)$ distribution on $\mathcal{S}_N^<$ and  the $\varepsilon$'s are independent, 
Rademacher distributed 
and independent of $(\rho_1, \cdots, \rho_n)$.
\item $\Vert X\Vert_p^p$ is Beta$(N/p ,1)$ distributed. In other words
\begin{equation*}
   \Vert X\Vert_p^p \stackrel{d}{=}\  U^{1/N}
\end{equation*}
where $U$ is uniform on $[0,1]$.
\end{enumerate}
\end{lemma}
1) is direct  via a change of variables and 2) is a genuine application of (\ref{grouping}).

\subsubsection{$p$-Generalized Dirichlet distributions}
\label{susuge}
Extending the previous lemma to generalized Dirichlet distributions is quite natural.
\begin{definition}
\label{deden}
Let $a_j, b_j , j = 1, \cdots ,N$ be positive real numbers. We say that  $X$ follows  the  $p$-Generalized Dirichlet distribution on $\mathbb{B}_{N,p}$
with parameter $({\mathbf a}, {\mathbf b})=(a_1,\cdots,a_N, b_1,\cdots, b_N)$  whenever
\begin{equation*}
    (X_1, \cdots , X_N) \stackrel{d}{=}\ \left(\varepsilon_1 \rho_1^{1/p}, \cdots , \varepsilon_N \rho_N^{1/p}\right)
\end{equation*}
where $(\rho_1, \cdots, \rho_N)$ follows the  $\mathcal{GD}_{\mathbf a, \mathbf b}$ distribution on $\mathcal{S}_N^<$ and  the $\varepsilon$'s are independent, 
Rademacher distributed and independent of $(\rho_1, \cdots, \rho_N)$.
\end{definition}

From the Section~\ref{subeta}, we deduce that the $p$-Generalized Dirichlet distribution with parameter $({\mathbf a}, {\mathbf b})$ has a density with respect to the non-normalized Lebesgue measure on $\mathbb{B}_{N,p}$
\begin{equation}
   \label{grobetan}
   H_{{\mathbf a}, {\mathbf b}}(x)= \frac{p^N}{\mathcal{Z}_{{\mathbf a}, {\mathbf b}}} \left(1 - \|x|_p^p\right)^{b_N -1} \prod_{j=1}^{N} \left(1-\|x^{(j-1)}\|_p^p\right)^{b_{j-1} -(a_j + b_j)}|x_j|^{pa_j -1}\,.
\end{equation}

Let us now go back to the parametrization of the balls developed in above section. Using equations \eqref{unun} and  \eqref{oninversec} in the context of $\ell ^p$ ball we recover the canonical representation:
\begin{eqnarray}
   C_1 &=& X_1\nonumber\\
   C_j &=& X_j \left(1 - \|X^{(j-1)}\|_p^p\right)^{-1/p},\quad j= 2, \cdots N.
    \label{ptozp}
\end{eqnarray}
and
\begin{eqnarray}
   X_1 &=& C_1\nonumber\\
   X_j &=& C_j \left[\prod_{k=1}^{j-1} (1- |C_k|^p)\right]^{1/p}, \quad j= 2, \cdots N.
    \label{ztopp}
\end{eqnarray}

Here we easily get
\begin{proposition}
\label{proppropren}
For $a_j>0 ,\; b_j> 0, j=1,\cdots,N$, assume that $X$ has the $p$-Generalized Dirichlet distribution with parameter
$({\mathbf a} , {\mathbf b})$ on $\mathbb{B}_{N,p}$ and set $C:=\mathcal{C}_N(X)= (C_1,\cdots,C_N)$ (see Lemma~\ref{lemjacob}). Then
\begin{itemize}
\item[i)] the random variables $C_1,\cdots,C_N$ are independent,
\item[ii)] for $j=1,\cdots,N$, $C_j \stackrel{d}{=}\ \varepsilon_j Z_j^{1/p}$, where $\varepsilon_j$ and $Z_j$ are independent, $\varepsilon_j$ has the Rademacher distribution, and
\begin{equation*}
\label{zlaw}
Z_j \stackrel{d}{=}\ \mathrm{Beta} (a_j, b_j)\,.
\end{equation*}
 In other words, the density of $C_j$ is
\begin{equation*}
   \label{densC}
   \mathbb{P}(C_j \in dx) =\frac{p\Gamma(a_j + b_j)}{2\Gamma(a_j)\Gamma(b_j)} |x|^{pa_j -1}\left(1-|x|^p\right)^{b_j -1} {\mathbf 1}_{(-1, 1)} (x) \!\ dx\,.
\end{equation*}
\end{itemize}
\end{proposition}

When the relation (\ref{relab}) is satisfied we have
\begin{equation}
\label{defdensH}
    H_{{\mathbf a}, {\mathbf b}}(x)= \frac{p^N}{\mathcal{Z}_{{\mathbf a}, {\mathbf b}}} \left(1 - \|x\|_p^p\right)^{b_N -1}
                                  \prod_{j=1}^{N} |x_j|^{pa_j -1}\,.
\end{equation}
In that case, the family of distribution may be indexed by $N+1$ parameters $(a_1, \cdots, a_N, b_N)$. This family is stable by permutation of the coordinates, which is quite interesting.

The following result is the counterpart of Theorem~\ref{unifreal}.
\begin{corollary}
\label{coro}
The density $H_{{\mathbf a}, {\mathbf b}}$ is  uniform  on $\mathbb{B}_{N,p}$ if, and only if, for every $j=1, \cdots , N$, $a_j = 1/p$ and $b_j = 1 + (N-j)/p$, or equivalently
\begin{align*}
   \nonumber
   Z_j &\stackrel{d}{=}\ \mathrm{Beta}\left(\frac{1}{p}, 1 + \frac{N-j}{p}\right),\\
   \mathbb{P}(C_j \in dx) &= \frac{p\Gamma\left(\frac{N-j+1}{p} + 1\right)}%
   {2 \Gamma\left(\frac{1}{p} + 1\right)\Gamma\left(\frac{N-j}{p} + 1\right)}\ \left(1 - |x|^p\right)^{\frac{N-j}{p}} \!\ {\mathbf 1}_{(-1, 1)}(x) \!\ dx\,.
\end{align*}
\end{corollary}
\medskip

\subsubsection{Some other sampling schemes}
\label{susuoth}
There are two other ways to sample uniformly in $\mathbb{B}_{N,p}$. They are issued from the polar decomposition. We give them for sake of completeness and because they will be useful in settling some limit theorems.

Let $\partial \mathbb{B}_{N,p}$ be  the  boundary of $\mathbb{B}_{N,p}$ (also called the sphere). The polar decomposition is the mapping
\begin{eqnarray*} \pi : \mathbb{R}^N \setminus \{0\} &\rightarrow& (0, \infty) \times\partial \mathbb{B}_{N,p}\\
\nonumber
x &\mapsto& \pi(x) = \left(\Vert x\Vert_p \ , \ \phi(x) := \frac{x}{\Vert x\Vert_p}\right)
\end{eqnarray*}
Let $\lambda_N$ be the Lebesgue measure of $\mathbb{R}^N$ and set
$$V_{N,p} := \lambda_N (\mathbb{B}_{N,p}) = 2^N \frac{\left(\Gamma\left(1 + \frac{1}{p}\right)\right)^N}{\Gamma\left(1 + \frac{N}{p}\right)}\,.$$
Let $\mu_B$  be the cone (probability) measure defined by
\begin{equation*}
    \mu_B (\Delta) = V_{N,p}^{-1} \ \lambda_N([0,1]\Delta)\; \;\;(\mbox{$\Delta$   measurable subset  of $\partial \mathbb{B}_{N,p}$}),
\end{equation*}
where  
$$[0,1]\Delta:=\{r \phi:\; r\in [0,1], \;\phi\in\Delta \}.$$
In other words, we have for any test function $f$
the decomposition (\cite{Naor1} Proposition 1)
$$\int_{\mathbb{R}^N} f(x) dx = NV_{N,p} \int_0^\infty r^{N-1} \int_{\partial \mathbb{B}_{N,p}} f(r\cdot \varphi) d\mu_B (\varphi) dr.$$
The probability $\mu_B$ is the image of the normalized Lebesgue measure on $\mathbb{B}_{N,p}$ by the mapping  $\phi$.
Notice that the cone probability and the surface measure are proportional if and only if $p=1$ or $2$.

\begin{lemma}
The pushing forward of $\mu_B$ by the mapping $$(x_1 , \cdots , x_N)\in \partial\mathbb{B}_{N,p} \mapsto (|x_1|^p, \cdots , |x_N|^p)\in \mathcal{S}_N$$ is Dir$_N(p^{-1})$. Moreover, for $1\leq k < N$,
\begin{equation*}\Vert x^{(k)}\Vert_p^p \stackrel{d}{=}\ \mathrm{Beta}\left(\frac{k}{p}, \frac{(N-k)}{p}\right)\,.\end{equation*}
\end{lemma}
The following lemma is well known.
\begin{lemma}
\label{norm}
If $\eta^{(N)} := (\eta_1, \cdots, \eta_N)$ has a distribution depending only upon
$\Vert \eta^{(N)}\Vert_p$ without any mass in $0$, then $\eta^{(N)}/\Vert \eta^{(N)}\Vert_p$ is independent of $\Vert \eta^{(N)}\Vert_p$ and its distribution is the cone probability $\mu_B$ on $\partial \mathbb{B}_{N,p}$.
\end{lemma}

Let  $\xi^{(N)} := (\xi_1, \cdots, \xi_N)$ be uniformly distributed in $\mathbb{B}_{N,p}$.
Let $(g_k)_{k\geq 1}$ be an i.i.d sequence having ${\mathbf G}_p$ distribution (see Section \ref{sec.dirichlet.world} ). Further, set
$G^{(N)} = (g_1, \cdots, g_N)$. There are two methods to generate a uniform sampling $\xi^{(N)}$ in the ball, both using a draw of $G^{(N)}$ and an extra independent variable.

$\bf A)$ According to  Lemma \ref{norm},
\begin{equation*}
    \phi^{(N)} := \frac{G^{(N)}}{\Vert G^{(N)}\Vert_p}
\end{equation*}
is independent of $\Vert G^{(N)}\Vert_p$ and is $\mu_B$ distributed  on $\partial \mathbb{B}_{N,p}$. From Lemma \ref{xixi0} 2), we claim as in \cite{Calafiorebis}  \footnote{The conic measure is misnamed there ``surface measure''} that
\begin{equation}
   \label{repres2}
   \xi^{(N)}\ \stackrel{d}{=}\ U^{1/N}\cdot \phi^{(N)},
\end{equation}
where $U$ is uniform on $[0,1]$ and sampled independently of $(g_k)_{k\geq 1}$.

$\bf B)$ In \cite{Barthe1}, it is proved that
\begin{equation*}
    \xi^{(N)} \stackrel{d}{=}\  \frac{G^{(N)}}{\left(\Vert G^{(N)}\Vert_p^p + Z\right)^{1/p}}
\end{equation*}
where $G^{(N)}$ is as above and  $Z$ is exponentially distributed and independent of $G^{(N)}$. Let us give a direct proof here since it is easy in our framework. Writing the righthand side as
$$\frac{G^{(N)}}{\left(\Vert G^{(N)}\Vert_p^p + Z\right)^{1/p}} = \frac{G^{(N)}}{\Vert G^{(N)}\Vert_p} \cdot \left(\frac{\Vert G^{(N)}\Vert_p^p}{\Vert G^{(N)}\Vert_p^p + Z}\right)^{1/p}\,,$$
we see from Lemma \ref{norm} that the two factors are independent.
From the additive property of Gamma distribution,  $\Vert G^{(N)}\Vert_p^p$
is  $\gamma(Np^{-1})$ distributed. So that, we can apply (\ref{10}) and get
\begin{equation*}
  \frac{\Vert G^{(N)}\Vert_p^p}{\Vert G^{(N)}\Vert_p^p + Z} \stackrel{d}{=}\ \mathrm{Beta}(Np^{-1}, 1)\,,
\end{equation*}
and taking the power $1/p$, this is exactly the required radial distribution (see Lemma~\ref{xixi0} point 2.).

Recently, various authors (\cite{Barthe1}, \cite{Naorseul}) were interested in the dependence structure of coordinates when sampling randomly in the unit ball or on the sphere. In this section, we give direct proofs of these results, carrying the known properties of Dirichlet distributions. In Lemma 2 of \cite{Naorseul}  (with references to analytical and geometric proofs in older papers),  Naor proved that the coordinates of a random sampling on the sphere are
negatively upper orthant dependent. To be more precise, let us recall two definitions (\cite{joagdev1983nar}).

\begin{definition}[Joag-Dev, Proschan]
\label{NAdef}
\begin{enumerate}
\item Random variables $U_1, \cdots, U_k$ are said to be negatively associated (NA) if
for every pair of disjoint subsets $A_1, A_2$ of $\{1, \cdots, k\}$
\begin{equation*}
    \hbox{Cov}\{ f_1(U_i, i \in A_1) , f_2(U_j , j \in A_2)\} \leq 0\,.
\end{equation*}
whenever $f_1$ and $f_2$ are increasing.
\item
Random variables $U_1, \cdots, U_k$ are said to be  negatively upper orthant dependent (NUOD), if
for all real $x_1, \cdots, x_k$,
\begin{equation*}
\label{NUODP}
     \mathbb{P}(U_i > x_i , i = 1, \cdots, k) \leq \prod_{i=1}^k \mathbb{P}(U_i > x_i)\,.
\end{equation*}
or equivalently, if
\begin{equation*}
\label{NUODF}
\mathbb{E}\left(\prod_{i = 1}^k f_i(U_i) \right) \leq \prod_{i=1}^k \mathbb{E}(f_i(U_i))\,.
\end{equation*}
for $f_1, \cdots, f_n$ increasing positive functions.
\end{enumerate}
\end{definition}
Some authors called property (\ref{NUODF}) sub-independence. NA is strictly stronger than NUOD.
\begin{proposition}
\label{DirNA}
If $\xi = (\xi_1, \cdots ,\xi_n)$ follows a Dirichlet distribution on $\mathcal{S}_n^<$ or $\mathcal{S}_n$, then the variables $\xi_1, \cdots ,\xi_n$ are NA.
\end{proposition}
This property is claimed for Dirichlet distributions (without precision) in (\cite{joagdev1983nar}).
Nevertheless, for the sake of completeness, we present a new  short proof here.

\begin{proof}
In a Dirichlet distribution the permutation of variables or the elimination of few of them yields still a Dirichlet distribution, thus it is enough to show that if $U \stackrel{d}{=}\ \hbox{Dir}_n(a_1, \cdots, a_n; a_{n+1})$ then  for $r <  n$ and, $f$ and $g$ increasing functions
\begin{equation*}
   \mathbb{E} \left[f(U_1, \cdots, U_r)g(U_{r+1}, \cdots, U_s)\right] \leq \mathbb{E} f(U_1, \cdots, U_r) \mathbb{E} g(U_{r+1}, \cdots, U_n)\,.
\end{equation*}
Conditioning by $S = U_1 + \cdots + U_k$ and using (\ref{quot})
it follows
\begin{equation*}
   \mathbb{E} \left[f(U_1, \cdots, U_r) g(U_{r+1}, \cdots, U_n)\right] = \mathbb{E} F(S) G(S)
\end{equation*}
with $F(S) = \mathbb{E}^S f(U_1, \cdots, U_r)$  and   $G(S) = \mathbb{E}^S g(U_{r+1}, \cdots, U_n)$ (here $\mathbb{E}^S$ denotes the conditional expectation with respect to $S$).
If we write $U$ as
\begin{equation*}
   U= \left(S\xi_1, \cdots, S\xi_r, (1-S) \xi_{r+1}, \cdots, (1-S)\xi_n\right)
\end{equation*}
it holds
\begin{eqnarray*}(\xi_1, \cdots, \xi_r) &\stackrel{d}{=}\ & \hbox{Dir}(a_1, \cdots , a_r)\\
 (\xi_{r+1}, \cdots, \xi_n) &\stackrel{d}{=}\ & \hbox{Dir}_{n-r}(a_{r+1}, \cdots , a_n; a_{n+1})\\
 S &\stackrel{d}{=}\ & \mathrm{Beta}(a_1 + \cdots + a_r ; a_{r+1}+ \cdots+ a_{n+1})
\end{eqnarray*}
and these three variables are independent. It follows then $F(s) = \mathbb{E} f(s\xi_1, \cdots, s\xi_r)$ and $G(s) =  \mathbb{E} g((1-s)\xi_1, \cdots, (1-s)\xi_r)$. The function $F$ is increasing and $G$ is decreasing, then applying the usual correlation inequality for a single variable yields
\begin{equation*}
\mathbb{E} F(S)G(S) \leq \mathbb{E} F(S) \mathbb{E} G(S)
\end{equation*}
and we obtain  $\mathbb{E} F(S) = \mathbb{E} f(U_1, \cdots, U_r)$ and as $1-S \stackrel{d}{=}\ \mathrm{Beta}(a_{k+1}+ \cdots+ a_{n+1} , a_1 + \cdots + a_r)$ we recover by product $\mathbb{E} G(S) = \mathbb{E} g(U_{k+1}, \cdots, U_n)$.
\end{proof}

In \cite{Perissinaki} and \cite{Barthe1} (see Theorem 6 and Lemma 5), the NUOD property was proved for the coordinates of a random sampling in the ball. Actually we can revisit these statements for a larger class of distributions and extend them to the NA property.

\begin{theorem}
\label{NAth}
\begin{enumerate}
\item If $X =(X_1, \cdots X_N)$ has the density of \eqref{defdensH} in the ball $\mathbb{B}_{N,p}$, then the variables $|X_1|, \cdots, |X_N|$ are NA.
\item If $X=(X_1, \cdots X_N)$ has the density
\begin{equation*}
  K(x) := \frac{1}{\mathcal{Z}_{\mathbf a}} \prod_{j=1}^{N} |x_j|^{pa_j -1}\,,
\end{equation*}
on the sphere $\partial\mathbb{B}_{N,p}$ (where $\mathcal{Z}_{\mathbf a}$ is the normalizing constant), then the variables $|X_1|, \cdots, |X_N|$ are NA.
\end{enumerate}
\end{theorem}
Let us stress that the uniform distribution in the ball (resp. on the sphere) satisfies the assumptions of the last theorem.
\medskip

\begin{proof}
Let $\xi_i = |X_i|^p$ for $i= 1, \cdots, N$ and $\xi = (\xi_1, \cdots, \xi_N)$.

1) When $X$ has the density  $H_{{\mathbf a},{\mathbf b}}$ of (\ref{defdensH}), a small change of variables yields
\begin{equation*}
   \xi \stackrel{d}{=}\ \hbox{Dir}_N(a_1, \cdots, a_N ; b_N).
\end{equation*}
We saw at  the end of Section \ref{direlem} that the variables $|\xi_i|$ are NA, and then the variables $|X_i|$ inherit the property.

2) We have
\begin{equation*}
(\xi_1 , \cdots, \xi_N) \stackrel{d}{=}\ \hbox{Dir}(a_1, \cdots, a_N)
\end{equation*}
and we get the same conclusion as above.
\end{proof}

\section{Asymptotics for $p$-generalized Dirichlet distributions}
\label{sec:asymptotics}
In this section, we use the previous representations to obtain asymptotic results (as $N$ goes to infinity) for the finite dimensional projections of random element of $B_{N,p}$.  We consider both convergence in distribution and large deviations.
\subsection{Convergence in distribution revisited}

\subsubsection{Poincar\'e-Borel for the uniform distribution on the ball}

With the representation \eqref{repres2} we get easily the Poincar\'e-Borel like lemma (see Lemma 1.2 of \cite{Ledoux} and  \cite{Diac} for historical account). As before,  $(g_k)_{k\geq 1}$ is a sequence of independent ${\bf G}_p$ distributed random variables and $G^{(N)} = (g_1, \cdots, g_N)$. For every $N \geq 1$, set
\begin{equation*}
   \phi^{(N)} = \frac{G^{(N)}}{\Vert G^{(N)}\Vert_p}\,.
\end{equation*}
Now, if $\eta^{(N)} := (\eta_1, \cdots, \eta_N)$ is uniformly distributed on the  $\ell_p$ sphere $\partial \mathbb{B}_{N,p}$, we have
\begin{equation*}
\label{quot2}
        N^{\frac{1}{p}} \eta^{(N)}\stackrel{d}{=}\ N^{\frac{1}{p}} \phi^{(N)} = \left(\frac {N}{\sum_{k=1}^N |g_k|^p}\right)^{1/p}  G^{(N)}\,.
\end{equation*}

If $\xi^{(N)} := (\xi_1, \cdots, \xi_N)$ is uniformly distributed in the ball $\mathbb{B}_{N,p}$, we have
\begin{equation*}
\label{quot3} N^{\frac{1}{p}} \xi^{(N)}
   \stackrel{d}{=}\ U^{1/N} N^{\frac{1}{p}} \phi^{(N)}
   \,.
\end{equation*}

By the strong law of large numbers,
\begin{equation*}
   \left(\frac {N}{\sum_{k=1}^N |g_k|^p}\right)^{1/p} \xrightarrow[N]{a.s.} \left(\mathbb{E}|g_1|^p\right)^{-1/p}= p^{1/p}.
\end{equation*}

Moreover $U^{N}$ converges to $1$ in distribution.
This yields a $p$-version of the classical form of the Poincar\'e-Borel's lemma:
\begin{proposition}\label{poincare.p.version}
If $k$ is fixed, and if $\pi^{(k)}$ denotes the projection on the $k$ first coordinates, then, as $N \rightarrow \infty$,
\begin{eqnarray*}
   \pi^{(k)}(N^{\frac{1}{p}} \eta^{(N)}) \xrightarrow{(d)} {\bf G}_p^{\otimes k},\\
   \pi^{(k)}(N^{\frac{1}{p}} \xi^{(N)}) \xrightarrow{(d)} {\bf G}_p^{\otimes k}.
\end{eqnarray*}
\end{proposition}

\subsubsection{Poincar\'e-Borel for generalized Dirichlet distributions}
\label{susePoing}
Let $X$ and $C$ be as in Proposition \ref{proppropren} and $k$ be a fixed positive integer. We will first give a result on the
asymptotic behavior, for large $N$, of $(N^{1/p}X^{(k)})$ ($k>0$ is fixed). Our proof uses the canonical representation of the ball and is quite simple.

\begin{theorem}
Let $k \geq 1$ be fixed. Assume that, for $j=1,\cdots,k$, $p b_j = N + o(N)$.
Then,
\begin{equation*}
   N^{1/p}X^{(k)} \xrightarrow[N]{(d)} \left(\varepsilon_1 Z_1^{1/p}, \cdots , \varepsilon_k Z_k^{1/p}\right)
\end{equation*}
where $\varepsilon_1, \cdots, \varepsilon_k, Z_1, \cdots, Z_k$ are independent, the $\varepsilon$'s are Rademacher distributed and for $j=1, \cdots, k$
\begin{equation*}
    Z_j\stackrel{d}{=}\ \gamma\left(a_j, \frac{1}{p}\right)\,.
\end{equation*} In other words $\varepsilon_j Z_j^{1/p}$ has the density
\begin{equation*}
    \mathbb{P}(\varepsilon_j Z_j^{1/p} \in dx)=\frac{p^{1-a_j}}{2\Gamma(a_j)}|x|^{pa_j -1}
    \exp(-{|x|^p}/{p})dx,\;\;(x\in\mathbb{R}).
\end{equation*}
\label{tpoing}
\end{theorem}

The proof is a straightforward consequence of the following useful lemma, whose proof follows the representation (\ref{10}) and the law of large numbers.
\begin{lemma}
Let $\theta > 0$ and $c(\theta)$ such that, as $\theta \rightarrow \infty$,
$\lim \frac{c(\theta)}{\theta}  = c > 0$. Then, for every $a > 0$,
\begin{equation*}
   \label{betainfini}
   \mathrm{Beta}(a, c(\theta))  \xrightarrow[\theta \rightarrow \infty]{~(d)~} \gamma(a, c)\,.
\end{equation*}
\end{lemma}

\begin{proof}[Proof of Theorem~\ref{tpoing}]
The above lemma yields for $j \leq k$
\begin{equation*}
    N^{1/p}C_j \xrightarrow[n]{~(d)~} Z_j\sim\gamma(a_j ,1/p)\,,
\end{equation*}
and by independence we get the convergence in distribution of $N^{1/p} C^{(k)}$.
Further using \eqref{unun} we may write $N^{1/p} X^{(k)} = N^{1/p} C^{(k)}+o_\mathbb{P}(1)$ and may conclude.
\end{proof}

\subsection{Large deviations}
Let us turn out to the large deviation companion theorem of Theorem \ref{tpoing}. For the basic notions on large deviations,
(definitions, contraction principle...) we refer to \cite{DZ}.
\begin{theorem}
\label{tgv} Under the assumption of Theorem \ref{tpoing}, $(X^{(k)})$ satisfies a large deviations principle (LDP) with
good rate function
\begin{equation*}
    I(x) = -\frac{1}{p} \ln\left(1-\|x\|_p^p\right),\;\;(x\in\mathbb{B}_{k,p}).
\end{equation*}
\end{theorem}

To prove this theorem we begin with the canonical variables.
\begin{proposition}
\label{LDPcanon}
Under the assumption of Theorem \ref{tpoing}, $(C^{(k)})$
satisfies a LDP with good rate function
\begin{equation*}
    J(c) = -\frac{1}{p} \ln\left( \prod_{i=1}^{k}(1-|c_i|^p)\right),\;\;
    c=(c_1,c_2,\cdots,c_k)\in (-1,1)^k.
\end{equation*}
\end{proposition}
The proof of this proposition is built on the following useful lemma.

\begin{lemma}
\label{PGDbeta}
Let $a > 0 , c > 0, \theta > 0$ and $c(\theta) > 0$ such that $c(\theta)/\theta \rightarrow c$ as $\theta \rightarrow \infty$. Let
$$Y_\theta \stackrel{d}{=}\ \mathrm{Beta} \left(a, c(\theta)\right)$$
Then, when $\theta \rightarrow \infty$, the family of distributions of $(Y_\theta)$ satisfies the LDP on $(0, 1)$ at scale $\theta$ with good rate function:
\begin{equation*}
\label{ldpbeta}
    J(x)  = - c \log (1 - x)\,.
\end{equation*}
\end{lemma}

This result appears with variants in the literature (Lemma~3.1 in \cite{DawsonFeng2}, Lemma~2.1 in \cite{fengseul}). It is proved therein with a direct computation. We give here another proof that enlightens the role played by the exponential distribution in all LDP about Beta distributions. A slight variant of this proof was used once in Lemma~4.3 in \cite{FabLoz}.
\medskip

\begin{proof}[Proof of Lemma \ref{PGDbeta}] We start with the representation (\ref{10}):
$$Y_\theta \stackrel{d}{=}\ \frac{\gamma(a)}{\gamma(a) + \gamma'(c(\theta))}\,,$$
where $\gamma(a)$ and $\gamma'(c(\theta))$ are independent. First, observe that the family of distributions of $\theta^{-1} \gamma(a)$ satisfies the LDP on $(0, \infty)$ with rate function $$I_0(x) = x\,.$$ Then, observe that the family of distributions of $\theta^{-1} \gamma\left(c(\theta)\right)$ satisfies the LDP on $(0, \infty)$ with rate function $$J_0 (x) = cx - 1 - \log (cx)\,.$$
By independence, the pair $(\theta^{-1} \gamma(a), \theta^{-1} \gamma'\left(b(\theta)\right))$ satisfies the LDP with rate function $(x_1, x_2) \mapsto I_0(x_1) + J_0 (x_2)$. Since the mapping
\begin{eqnarray*}
(0, \infty) \times (0, \infty) &\rightarrow& (0,1)\\
(x,y) &\mapsto& \frac{x}{x+y}
\end{eqnarray*}
is continuous, the contraction principle gives an LDP for $(Y_\theta)$ with rate function
\begin{eqnarray*}J(x) = \inf\left\{I_0(x_1) + J_0 (x_2) ; \frac{x_1}{x_1 + x_2} = x \right\}
= - c \log (1 - x)\,,
\end{eqnarray*}
which ends the proof.
\end{proof}
\medskip

\begin{proof}[Proof of Theorem~\ref{LDPcanon}]
Since the coordinates of $C^{(k)}$ are independent the LDP will follow
from the LDP of each $C_j$ $(j=1,2,...,k)$.
From Proposition~\ref{proppropren} observe that
$$C_j \stackrel{d}{=}\ \epsilon_j Z_j^{1/p} \ , \ \ Z_j \stackrel{d}{=}\ \mathrm{Beta}(a_j , b_j)\,.$$
For $j$ fixed, we may apply Lemma \ref{PGDbeta} with
\begin{equation*}a= a_j \ , \theta=N , \ b(\theta) = b_j  , c = \lim \frac{b_j}{N} = \frac{1}{p}, \,.\end{equation*}
We get an LDP for $Z_j$ on $(0,1)$ with rate function $x \mapsto -p^{-1} \log (1-x)$. One easily deduces from this that
$C_j$ satisfies the LDP on $(-1, +1)$ with rate function $x \mapsto -p^{-1} \log (1-|x|^p)$. By independence, the vector $C^{(k)}$ satisfies the LDP on $(0,1)^k$ with rate function
\begin{equation*}
(c_1, \cdots , c_k) \mapsto \sum_{j=1}^k -p^{-1} \log (1-|c_k|^p)\,.
\end{equation*}
\end{proof}
\medskip

\begin{proof}[Proof of Theorem~\ref{tgv}]
LDP for $X^{(k)}$ follows directly from the  contraction
principle. The good rate function is given by the
relationship
\begin{equation*}
   I(x) = J(\mathcal{C}_k(x)),\; x\in \mathbb{B}_{k,p}.
\end{equation*}
Furthermore, from the definitions of the canonical coordinates (see
\eqref{unun}) it follows that
\begin{equation*}
    \prod_{i=1}^{k}(1-|c_i|^p) = 1-\|x\|_p^p\,,
\end{equation*}
which allows to conclude the proof.
\end{proof}

\paragraph{The $\ell^2$-ball and a functional LDP}
The $\ell^2$-ball case is quite peculiar, due to its connection with
 the functional unit ball
 \begin{equation*}
    \mathcal{B}_2:=\{f\in L^2([0,1]):\|f\|_2 < 1\}\,,
\end{equation*}
of the Hilbert space $L^2([0,1])$. This allows to extend somehow the LDP of Theorem \ref{tgv}.
In this section, we will give a functional LDP companion result of  Theorem \ref{tgv}.
Let $(e_n)_{n\geq1}$ be any orthonormal basis of $L^2([0,1])$. So that, we may rewrite
\begin{equation*}
   \mathbb{B}_{N,2} := \bigg\{(x_1, x_2,...,x_N)\in\mathbb{R}^N:
             x_i = \langle f,e_i\rangle,\ i=1,2,...,N,\ f\in \mathcal{B}_2\bigg\}.
\end{equation*}
Let the random sequence $(F_N)$ of $\mathcal{B}_2$ be defined by
\begin{equation*}
     F_N = \sum_{i=1}^{N} X_ie_i,
\end{equation*}
where the sequence $(X_i)$ satisfies the assumption of Theorem \ref{tpoing}.
The following Theorem follows directly from a projective limiting  procedure (see \cite{DZ} Section 4.6).
\begin{theorem}
The sequence $(F_N)_N$ satisfies  a LDP in $\mathcal{B}_2$ with good rate function
\begin{equation*}
    I_\mathcal{B}(f) = - \frac{1}{2} \ln(1-\|f\|_2)
    \,.
\end{equation*}
\end{theorem}

\subsection{Donsker limit theorems}

Come back to Proposition~\ref{poincare.p.version}, in a first extension we can take $k= k(N)$ and show that the distance in variation between the law of $\pi^{(k)}(N^{\frac{1}{p}} \eta^{(N)})$ and ${\bf G}_p^{\otimes k}$ tends to $0$ as soon as $k(N) = o(N)$ \cite[Theorem 3]{Naor1}. The Euclidean case ($p=2$) is treated in  \cite{Diac} (see \cite{Jiang3} for related results).
Besides, since we may write for every $N \geq 1$
$$p^{-1/p} N^{\frac{1}{p}-\frac{1}{2}}\left(\sum_{k=1}^{\lfloor Ns\rfloor} \eta_k , s \in [0,1]\right) \ \stackrel{d}{=}\
\left(\frac {p^{-1}N}{\sum_{k=1}^N |g_k|^p}\right)^{1/p} \cdot \left(N^{-\frac{1}{2}} \sum_{k=1}^{\lfloor Ns\rfloor} g_k , s \in [0,1]\right)\,,$$
 we can deduce the convergence to the standard Brownian motion $\{W_s , s \in [0,1]\}$. It is the classical Donsker's theorem for self-normalized processes. It holds actually under very weak assumptions (see \cite{Donsker}).
In the same vein, owing to the results of Shao (\cite{Shao2}) the family $\frac{N^{\frac{1}{p} - 1}}{p}\left(\sum_{k=1}^N \eta_k \right)$  satisfy the LDP with good rate function
\begin{equation*}
    I(x)= \inf_{c \geq 0}\sup_{t \geq 0}\left\{t |x|(p-1) c^{\frac{p}{p-1}} - \log \int \frac{\exp\left(t\left(cy - (1+|x|) |y|^p \right)\right)}{2\Gamma\left(1 +\frac{1}{p}\right)}\ dy \right\}
\end{equation*}

One may think that the LDP holds also for the sequence of processes
$\left(\frac{N^{\frac{1}{p} - 1}}{p}\sum_{k=1}^{\lfloor Ns\rfloor} \xi_k , s \in [0,1]\right)$
with the rate function
\begin{equation*}
   \widetilde I (\varphi) = \int_0^1 I(\dot\varphi(s)) ds\,.
\end{equation*}

\subsection{The $\ell^1$ ball and the GEM$(\alpha, \theta)$ distribution}
In a nice paper, Dawson and Feng (\cite{DawsonFeng2}, Theorem 4.3)  proved that the LDP holds in
$\mathbb{B}_{N,1}$ when the underlying canonical variables are Beta $(1, \theta)$ and let $\theta \rightarrow \infty$. It is exactly a particular case of our Theorem \ref{tgv} with $p=1$, $\beta_j =0$ and $\alpha_j = j/p$ and $\theta = N/p$.
In their Theorem 4.4, they extend the LDP to $\mathbb{B}_{\infty,1}$ with rate function defined on $\mathcal{S}_\infty^<$ by
\begin{equation*}
    I_1(x) =  - \log (1 - \Vert x\Vert_1).
\end{equation*}
It is the so called GEM$(\theta)$ model. It is exactly a particular case of our Theorem \ref{tgv} with $p=1$, $\beta_j =0$ and $\alpha_j = j -1$ and $\theta = N$.

In another paper Feng \cite{fengseul} proved the LDP when the $k$-th canonical variable is Beta$(1-\alpha, \theta + k\alpha)$ distributed. They obtained the same rate function. It is the so called GEM$(\alpha, \theta)$ model.  It is exactly a particular case of our Theorem \ref{tgv} with $p=1$, $\beta_j =-\alpha$ and $\alpha_j = (\alpha +1)j -1$ and $\theta = N$.

\section{Moment spaces revisited}
\label{sec.mom.ball}

\subsection{Moments: the complex case}
\label{sec.momens.complex}
All this subsection comes from the book of Simon \cite{Simon1} Section 1 or \cite{Simon3} Sections 2 and 3. We recall here the connection between moments of a probability measure on the torus $\mathbb{T}$ and canonical moments built through orthogonal polynomials. To begin with, let $\mu$ be an arbitrary nontrivial (that is not supported by a finite number of points) probability on $\mathbb{T}$. The functions $1, z, z^2, \cdots$ are linearly independent in $L^2(\mathbb{T}, d\mu)$. 
Following the Gram-Schmidt procedure we define the monic
orthogonal polynomials $(\Phi_n)$. More precisely, $\Phi_0 (z) \equiv 1$ and for $n\geq 1$, $\Phi_n (z)$
is the projection of $z^n$ onto $\{1, \cdots,z^{n-1}\}^\perp$.
If $\mu$ is supported on the finite set $\{z_1, \cdots , z_N\}$, we still define $\Phi_k$ until $k= N-1$. We define $\Phi_N$ as the unique monic polynomial of degree $N$ such that $\Vert \Phi_N\Vert =0$ i.e.
$$\Phi_N (z) = \prod_{j=1}^N (z-z_j)\ \,.$$
Some useful polynomials associated to the sequence $(\Phi_n)$ are
the reversed (or reciprocal) polynomials. They are defined by $\Phi_0^\star (z)\equiv 1$ and
\begin{equation*}
\label{reversed}
    \Phi_n^\star (z) = z^n \overline{\Phi_n(1/\bar z)}.
\end{equation*}
Notice that $\Phi_n^\star$ is the unique polynomial of degree at most $n$,
orthogonal to $z, z^2, \cdots, z^n$ and such that $\Phi_n^\star (0)=1$. We now define
a quantity which appears to be central in our paper.
\begin{definition}
For $j\in\mathbb{N}$, we define the canonical moment $c_j:=-\overline{\Phi_{j}(0)}$.
\label{defican}
\end{definition}
In the sequel, when it will be necessary to make precise that the canonical moment depend on the underlying measure $\mu$, we will sometimes also write $c_j(\mu)$.
The coefficients  $c_j$, $j \geq 0$ are called  Verblunsky coefficients  by Simon. They are also named
after Schur, Szeg\H{o}, or Geronimus \cite{ibra}. They are sometimes called reflection coefficients \cite{burg}. One of
their properties is recalled below without proof for further use
\begin{proposition}
\label{propcan}
\begin{equation*}
     \Vert\Phi_{n+1}\Vert^2 = \left(1 -|c_{n+1}|^2\right)\Vert\Phi_{n}\Vert^2 = \prod_{j=1}^{n+1}\left(1 -|c_j|^2\right).
     \label{moreover}
\end{equation*}
\end{proposition}
Consequently, if $\mu$ is nontrivial, $c_j\in \mathbb{D}$ for every $j> 0$.
Further, if the support of $\mu$ consists in $N$ points, then $c_j\in \mathbb{D}$ for $1\leq j\leq N-1$ and $c_{N} \in \partial\mathbb{D}$.
\label{211}
A theorem due to Verblunsky asserts that the  correspondence between $\mu$ and the sequence of its canonical moment is a bijection.
The Verblunsky's formula (\cite{Simon1} Theorem 1.5.5) claims that for each $N$, there is a polynomial
$V^{(N)}(c_1, \cdots, c_{N-1}, \bar{c}_{0}, \cdots \bar{c}_{N})$ with integer coefficients so that the
moments $\{t_n\}_n$ of $\mu$ satisfy
\begin{equation}
t_{N}:=\int z^N\mu(dz) = c_N \prod_{j=1}^{N-1}\left(1 -|c_j|^2\right)
       +  V^{(N)}(c_1, \cdots, c_{N-1}, \bar{c}_{0}, \cdots \bar{c}_{N-1})\,.
       \label{canon}
\end{equation}
Conversely, $c_N$ is a rational function of $t_1, \bar t_1, \cdots , t_{N-1}, \bar t_{N-1}, t_{N}$.
Moreover, as remarked by Simon (\cite{Simon1} Section 3.1), formula (\ref{canon}) tells us that canonical moments measure
relative positions of $t_n$ among all values consistent with $c_0, c_1, \cdots, c_{n-1}$.
To be more precise, for $n \geq 1$, set
\begin{equation*}
     M_N^{\mathbb{T}} = \left\{\left(\int z^j \mu(dz)\right)_{1\leq j\leq N} : \mu \in \mathcal{M}_1(\mathbb{T})\right\}\, ,
\end{equation*}
where $\mathcal{M}_1(\mathbb{T})$ denotes the set of all probability measures on $\mathbb{T}$.
Then, given   $\left(t_1, \cdots , t_N\right)\in M_N^{\mathbb{T}}$, the range of the $(N+1)$th moment
\begin{equation*}
   t_{N+1} = \int z^{N+1} d\eta(z)
\end{equation*}
as $\eta$ varies over all probability measures having $(t_1, \cdots , t_N)$ as $N$ first moments,
is a disk centered at $s_{N+1} = V^{(N)}(t_1, \cdots, t_{N}, \bar t_1, \cdots \bar t_{N})$ with radius
\begin{equation*}
   r_{N+1} = \prod_{j=1}^{N}\left(1 -|c_j|^2\right)
\end{equation*}
(by Verblunsky theorem, these quantities only depend on the prescribed $N$ first moments).
If $r_{N+1}\neq 0$, the relative position is
\begin{equation*}
   \frac{t_{N+1} - s_{N+1}}{r_{N+1}}\in \mathbb{D}.
\end{equation*}
A very nice result is that the above quantities are the canonical moments $c_{N+1}$ of $\mu$ (see \cite{DS}). So that, as pointed out in \cite {DS}, canonical moments may be built both geometrically or algebraically.

\subsection{Moment space and generalized Dirichlet distribution}
\label{ssec:mom_spc_GD}

In this section we discuss the connection between randomized balls and randomized moment spaces.
Indeed, the asymptotic results in \cite{Chang}, \cite{FabLoz} and \cite{LozEJP} are in the same spirit as those obtained here. In an early version of the present paper, we wrote a result exhibiting a natural way to push forward the uniform measure on complex moment spaces towards the uniform one on the complex Euclidean balls. Unfortunately, this nice result is not true! It was based on a wrong Jacobian computation for complex moment space performed in Lemma 7.3 in \cite{LozEJP} where a factor $2$ is missed in the exponentiation. 

\paragraph{Complex moments and Euclidean balls}
\label{ballsand}
The aim of this subsection is to underline a natural connection between the moment space $M_N^{\mathbb{T}}$ and the Euclidean ball. This will done using canonical moments. To begin with, let us go back to the sequence of orthogonal polynomials for the probability measure $\mu$. Let $N$ be an integer such that the support of $\mu$ has cardinality at least $N+1$. Here, we will normalize the orthogonal system in a different way than in Section~\ref{sec.momens.complex} by setting:
\begin{equation*}
   \label{orthonc}
   \varphi_0 = 1 , \quad \varphi_n = \frac{\Phi_n}{\Vert\Phi_n\Vert}, \quad n = 1,2,...,N.
\end{equation*}
We also define the associated reversed polynomials:
\begin{equation*}
   \varphi_n^\star = \frac{\Phi_n^\star}{\Vert\Phi_n\Vert}\,.
\end{equation*}
Since $\varphi_{N}^\star$ is a polynomial of degree $N$  and since
$\Vert \varphi_{N}^\star \Vert^2 = \Vert \varphi_{N} \Vert^2 = 1$, we have
\begin{equation}
   \label{=1}
   \sum_{k=0}^{N} |\langle \varphi_N^\star , \varphi_{k} \rangle|^2 = 1\,.
\end{equation}
Set
\begin{equation}
   \label{defpi}
   \pi_k := \langle \varphi_{N}^\star , \varphi_{k} \rangle,\quad k = 0,1,...,N
\end{equation}
the formula (1.5.59) in \cite{Simon1} yields to the following lemma.
\begin{lemma} \label{lSimon63}
\begin{equation*}
\label{Simonp63}
      \pi_k  = - \bar c_{k-1} \prod_{r = k+1}^{n} \sqrt{1 - |c_{r}|^2},\quad k = 1,2,...,N,
\end{equation*}
where by convention $c_0 = -1$.
\end{lemma}
For the sake of completeness let us give a short proof.

\begin{proof}
Let $P$ be a polynomial of degree at most $N$.
Recall that $\Phi_N^\star$ is orthogonal to $z,\ldots, z^{N}$ so that  $\langle \Phi_n^\star , [P(z) - P(0){\bf 1}] \rangle = 0$. Therefore
\begin{equation*}
       \langle \Phi_N^\star , P\rangle = P(0) \int \overline{\Phi_N^\star (z)} d\mu(z) = P(0) \int \bar z^N \Phi_N(z) d\mu(z) = P(0) \Vert \Phi_N\Vert^2
\end{equation*}
hence, taking $P = \varphi_k$  (for $k\leq N)$
\begin{equation*}
    \pi_k =\langle  \varphi_N^\star , \varphi_k \rangle = \left\langle  \frac{\Phi_N^\star}{\Vert\Phi_N\Vert} , \varphi_k
           \right\rangle
          = \varphi_k(0)\Vert\Phi_n\Vert
          = \Phi_k (0) \frac{\Vert \Phi_n\Vert}{\Vert\Phi_k\Vert}.
\end{equation*}

The previous equality and Proposition \ref{propcan} give
\begin{equation*}
 \pi_0=\Vert\Phi_N\Vert=\prod_{r = 1}^{N} \sqrt{1 - |c_{r}|^2}.
\end{equation*}
Now for $1 \leq k\leq N$, by the same arguments we obtain
\begin{equation*}
   \pi_k = \Phi_k (0) \frac{\Vert \Phi_N \Vert}{\Vert\Phi_k\Vert}
         = -\bar c_k\frac{\Vert \Phi_N\Vert}{\Vert\Phi_k\Vert}
         = -\bar c_k\frac{\prod_{r = 1}^{N} \sqrt{1 - |c_{r}|^2}}
           {\prod_{r = 1}^{k} \sqrt{1 - |c_{r}|^2}}=-\bar c_k \prod_{r = k+1}^{N} \sqrt{1 - |c_{r}|^2}.
\end{equation*}
\end{proof}

Now, from (\ref{defpi}) and (\ref{=1}) we see that the point
\begin{equation}
   \label{reverse}
   z = (z_1, \cdots, z_N) := (\pi_{N} , \cdots , \pi_1)
\end{equation}
lies in the complex ball $\mathbb{B}_{N,2}^{\mathbb{C}}$ (see \eqref{defboulecomplex} below).
Furthermore,  setting:
\begin{equation*}\label{kappazeta}
    \kappa_r = - \bar c_{N+1-r},\quad r=1,2,\ldots,N
\end{equation*}
we get from (\ref{reverse})  and Lemma \ref{lSimon63}
\begin{eqnarray*}
   \label{newkx}
   z_1 &=& \kappa_1\\
   z_j &=& \kappa_j \prod_{s=1}^{j-1} {\sqrt{1 - |\kappa_s|^2}},\quad j = 1,2,...,N.
\end{eqnarray*}
Roughly speaking, the previous relation is the analogue for the ball of (\ref{canon}). Notice that
the relationships $(t_k)\leftrightarrow(c_k)$ and $(z_k)\leftrightarrow(\kappa_k)$ are both triangular
and bijective. They measure the relative position of a coordinate knowing the previous ones.

\paragraph{The real case.} Finally, we discuss a result in the real case. 
Consider now the following diagram
\begin{equation*}
    \begin{array}{ccc}
         M^{[0,1]}_N & \overset{\Sigma_N}{\longleftrightarrow} & \mathbb{B}_{N,2} \\
         \Big\updownarrow\ \text{{\scriptsize(i)}} &   & \Big\updownarrow\ \text{{\scriptsize(iii)}}\\
         (0,1)^N & \overset{\text{(ii)}}{\longleftrightarrow}  & (-1,1)^N
    \end{array}
\end{equation*}
where (i) is the canonical moments transformation, (ii) is the coordinatewise transformation $t\mapsto 2t-1$, (iii) is the inverse canonical coordinate transformation and $\Sigma_N$ is obtained by composition of these transformations.

Obviously, using Theorem~\ref{unifreal} and Proposition~\ref{proppropren}, the pushforward of the uniform probability measure on  $M^{[0,1]}_N$ by $\Sigma_N$ leads to the generalized Dirichlet Distribution on $\mathbb{B}_{N,2}$  on $\mathbb{B}_{N,2}$ with $a_1=a_2=\cdots=a_N = 1/2$ and $b_j = N-j+1$, $j=1,2,...,N$.

\end{document}